\documentclass[11pt]{amsart}

\usepackage[left=1in,right=1in,top=1in,bottom=1in]{geometry}

\usepackage{microtype}  

\usepackage{mathtools}
\usepackage{bbm}  

\swapnumbers

\usepackage{hyperref}

\usepackage{amssymb,amsfonts,amsmath,amsthm}
\usepackage[mathscr]{eucal}
\usepackage[alphabetic]{amsrefs}
\usepackage{stmaryrd}

\usepackage[ps,matrix,arrow,curve,cmtip]{xy}

\title{Classifying spaces for $1$-truncated compact Lie groups}
\author{Charles Rezk}
\date{ \today}
\address{Department of Mathematics \\
University of Illinois \\ 
Urbana, IL}
\email{rezk@illinois.edu}
\thanks{The author was supported under NSF grant DMS-1406121.}

\numberwithin{equation}{section}

\makeatletter
  \let\c@subsection\c@equation
\makeatother

\theoremstyle{plain}   

\newtheorem{thm}[subsection]{Theorem}
\newtheorem{prop}[subsection]{Proposition}
\newtheorem{cor}[subsection]{Corollary}
\newtheorem{lemma}[subsection]{Lemma}

\theoremstyle{remark}
\newtheorem{rem}[subsection]{Remark}    
\newtheorem{exam}[subsection]{Example}

\theoremstyle{plain}

\begin{document}


\newcommand{\margnote}[1]{\mbox{}\marginpar{\tiny\hspace{0pt}#1}}

\def\lambada{\lambda}


\newcommand{\id}{\operatorname{id}}
\newcommand{\colim}{\operatorname{colim}}
\newcommand{\llim}{\operatorname{lim}}
\newcommand{\Cok}{\operatorname{Cok}}
\newcommand{\Ker}{\operatorname{Ker}}
\newcommand{\Image}{\operatorname{Im}}
\newcommand{\op}{{\operatorname{op}}}
\newcommand{\Aut}{{\operatorname{Aut}}}
\newcommand{\End}{{\operatorname{End}}}
\newcommand{\Hom}{{\operatorname{Hom}}}

\newcommand*{\ra}{\rightarrow}
\newcommand*{\lra}{\longrightarrow}
\newcommand*{\xra}{\xrightarrow}
\newcommand*{\la}{\leftarrow}
\newcommand*{\lla}{\longleftarrow}
\newcommand*{\xla}{\xleftarrow}

\newcommand{\ho}{\operatorname{ho}}
\newcommand{\hocolim}{\operatorname{hocolim}}
\newcommand{\holim}{\operatorname{holim}}

\newcommand*{\realiz}[1]{\left\lvert#1\right\rvert}
\newcommand*{\len}[1]{\left\lvert#1\right\rvert}
\newcommand{\set}[2]{{\{\,#1\mid#2\,\}}}
\newcommand*{\tensor}[1]{\underset{#1}{\otimes}}
\newcommand*{\pullback}[1]{\underset{#1}{\times}}
\newcommand*{\powser}[1]{[\![#1]\!]}
\newcommand*{\laurser}[1]{(\!(#1)\!)}
\newcommand{\ndiv}{\not|}
\newcommand{\pairing}[2]{\langle#1,#2\rangle}

\newcommand{\F}{\mathbb{F}}
\newcommand{\Z}{\mathbb{Z}}
\newcommand{\N}{\mathbb{N}}
\newcommand{\R}{\mathbb{R}}
\newcommand{\Q}{\mathbb{Q}}
\newcommand{\C}{\mathbb{C}}

\newcommand{\point}{{\operatorname{pt}}}
\newcommand{\Map}{\operatorname{Map}}
\newcommand{\eev}{\wedge}
\newcommand{\sm}{\wedge} 

\newcommand*{\mc}{\mathcal}
\newcommand*{\mf}{\mathfrak}
\newcommand*{\mr}{\mathrm}
\newcommand*{\mb}{\mathbb}
\newcommand*{\ul}{\underline}
\newcommand*{\ol}{\overline}
\newcommand*{\wt}{\widetilde}
\newcommand*{\wh}{\widehat}

\newcommand{\dfn}{\textbf}

\def\noloc{\;{:}\,}

\def\defeq{\overset{\mathrm{def}}=}

\newcommand{\forcepar}{\mbox{}\par}

\newcommand{\ad}{\operatorname{ad}}

\newcommand{\Top}{\mathrm{Top}}
\newcommand{\red}{\mathrm{red}}

\newcommand*{\Len}[1]{\left\lVert#1\right\rVert}
\newcommand{\Sk}{\operatorname{Sk}}

\newcommand{\eMap}{\ul{\mathrm{Map}}}
\newcommand{\Fun}{\mathrm{Fun}}

\begin{abstract}
A 1-truncated compact Lie group is any extension of a finite group by
a torus.  In this note we compute the homotopy types of
$\Map_*(BG,BH)$,  $\Map(BG,BH)$, 
and $\Map(EG, B_GH)$ for compact Lie groups $G$ and $H$ with $H$
1-truncated, showing 
that they are computed entirely in terms of 
spaces of homomorphisms from $G$ to $H$.  These results generalize
the
well-known case when $H$ is finite, and the case of $H$ compact
abelian due to Lashof, May, and Segal.
\end{abstract}

\maketitle


\section{Introduction}

By a \dfn{1-truncated} compact Lie group $H$, we mean one whose
homotopy groups vanish in dimensions 2 and greater.  Equivalently, $H$
is a compact Lie group with identity component $H_0$ a torus
(isomorphic to some $U(1)^d$); i.e., an extension of a finite group by
a torus.

The class of 1-truncated compact Lie groups includes (i) all finite groups,
and (ii) all compact abelian Lie groups, both of which are included in the
class (iii) all groups which
are isomorphic to a product of a compact abelian Lie group with a finite
group, or equivalently a product of a torus with  a finite group.

The goal of this paper is to extend certain results, which were
already known for finite groups, compact abelian Lie groups, or
products thereof, to all 1-truncated compact Lie groups.

We write $\Hom(G,H)$ for the space of continuous homomorphisms,
equipped with the compact-open topology.
Our first theorem relates this to the space of based maps between
classifying spaces.
\begin{thm}\label{thm:main}
For $G,H$  compact Lie groups with  $H$ 1-truncated,
the evident map 
\[
B\colon \Hom(G,H) \ra \Map_*(BG, BH)
\]
is a weak equivalence.   
\end{thm}

Using this, we will derive an unbased variant.
\begin{thm}\label{thm:unbased}
For $G,H$  compact Lie groups with $H$   1-truncated, there is a weak  
equivalence
\[
\Hom(G,H)\times_H EH \ra \Map(BG,BH).
\]
Here $H$ acts on $\Hom(G,H)$ by conjugation: $h\cdot \phi= h\phi
h^{-1}$.  
\end{thm}

When $H$ is discrete, these are well-known and classical results.  
The case of $H$  an abelian compact Lie group is proved in
\cite{lashof-may-segal-equivariant-abelian}; both the finite and
compact abelian Lie cases are discussed in
\cite{may-remarks-equivariant-bundles}*{Thm.\ 5 and Thm.\ 10}.  

\begin{rem}\label{rem:unbased-in-terms-of-cent}
For $G$ and $H$ compact, there is a homeomorphism
\eqref{prop:hom-space-topology} 
\[
\Hom(G,H) \approx \coprod_{[\phi\colon G\ra H]} H/C_H(\phi)
\]
where the coproduct is over conjugacy classes of homomorphisms, and
$C_H(\phi)$ is the centralizer of $\phi(G)$ in $H$.  When $H$ is a
1-truncated compact Lie group, we see from \eqref{thm:main} that
$\Map_*(BG,BH)$ is therefore
weakly equivalent to this coproduct, and from \eqref{thm:unbased} that
there is a weak equivalence
\[
\Map(BG,BH) \approx \coprod_{[\phi\colon G\ra H]} BC_H(\phi). 
\]
\end{rem}

Finally, we will give a description of the fixed points of the
equivariant classifying space $B_GH$, which represents $G$-equivariant
$H$-principal bundles, in the case that $G$ and $H$ are compact Lie
and $H$ is 1-truncated.

\begin{thm}\label{thm:equivariant}
For $G, H$ compact Lie groups with $H$ 1-truncated, the map
\[
\pi^*\colon B_G H \ra \Map(EG, B_GH)
\]
induced by restriction along $\pi\colon EG\ra *$ is a $G$-equivariant weak
equivalence.  
\end{thm}
The case of $H$ finite or compact abelian is proved in
\cite{may-remarks-equivariant-bundles}. 

\begin{rem}
For any closed
subgroup $G'\leq G$, taking $G'$ fixed points gives rise to  maps
\[
(B_GH)^{G'} \ra \Map(EG,B_GH)^{G'} \xla{\sim} \Map(EG, BH)^{G'} \approx
\Map(BG',BH),
\]
and \eqref{thm:equivariant} amounts to saying that for any $G'$ the
first map in this sequence is a weak 
equivalence.  (The middle map arises from a $G$-equivariant weak
equivalence $\Map(EG,BH)\ra \Map(EG,B_GH)$; see the proof of
\eqref{lemma:htpy-commutative-classifying}.) 
It is standard \cite{lashof-may-gen-equivariant-bundles}*{Theorem 10}
that, for arbitrary compact $G$ and $H$,  $(B_GH)^{G'}$ 
is weakly equivalent to $\coprod_{[\phi\colon G'\ra H]} BC_H(\phi)$,
while if $H$  is also 1-truncated, \eqref{thm:unbased} and
\eqref{rem:unbased-in-terms-of-cent} 
imply that $\Map(BG',BH)$ is also weakly equivalent to the same
coproduct, thus giving an abstract weak equivalence
$(B_GH)^{G'}\approx \Map(BG',BH)$.  The point of
\eqref{thm:equivariant} is to show that the map $\pi^*$ exhibits this
equivalence. 
\end{rem} 

The map of \eqref{thm:equivariant} in a certain sense classifies the
formation of the $G$-Borel quotient.  That is, given a $G$-equivariant
map $f\colon X\ra B_GH$ classifying a $G$-equivariant principal
$H$-bundle $P\ra X$, the $G$-equivariant map $\pi^*f\colon X\ra
\Map(EG,B_GH)$ is  adjoint to a non-equivariant map $X\times_G EG\ra
B_GH$ which classifies the bundle $P\times_G EG\ra X\times_G EG$; see
\cite{may-remarks-equivariant-bundles}.   As a consequence of the
theories of classifying spaces, we obtain
the following.
\begin{cor}
Let $G$ and $H$ be compact Lie groups with $H$ 1-truncated.  Then for
a paracompact 
$G$-space $X$,
formation of $G$-Borel quotient gives 
rise to a bijection between (i) equivalence classes of $G$-equivariant
principal $H$-bundles over $X$, and (ii) equivalence classes of
principal $H$-bundles over $X\times_G EG$.  
\end{cor}

\begin{rem}
For comparison,  we note that there are well-known results (stemming
from work of Dwyer-Zabrodsky
\cite{dwyer-zabrodsky-maps-between-classifying} and Notbohm
\cite{notbohm-maps-between-classifying}) on spaces of maps \emph{from}
(rather than \emph{to}) the classifying space of a $p$-toral group
(a $p$-toral group is  an extension of a \emph{finite $p$-group} by a torus).
For instance, 
\cite{notbohm-maps-between-classifying}*{Thm.\ 1.3} may be interpreted
as saying that for an arbitrary compact Lie group $H$ and $p$-toral
$G$, the map $B_GH\ra \Map(EG,B_GH)$ (cf.\ \eqref{thm:equivariant})
induces an isomorphism in mod $p$ homology on fixed points for all
closed subgroups of $G$ (this interpretation is given as
\cite{may-remarks-equivariant-bundles}*{Thm. 9}.)
\end{rem}

\subsection{Organization of the paper}

The proof of \eqref{thm:main} is the probably the most interesting
part of  the paper.  It is carried out in
\S\S\ref{sec:nerve-crossed-module}--\ref{sec:proof-based}.  The key
ingredient is the use of the nerve $N(H,V)$ of the ``exponential
crossed module'' \eqref{exam:toral-crossed-module} of the 
1-truncated compact Lie group $H$.  We first show that the simplicial space
$N(H,V)$ is a Reedy 
fibrant model for the usual simplicial nerve $NH$ of $H$
\eqref{cor:reduction-to-nerve}, and 
so can be used to compute maps $BG$ to  $BH$, in
terms of maps of simplicial spaces from $NG$ to $N(H,V)$.  The proof
is completed (\S\ref{sec:proof-based}) by showing that, in a certain
sense, the difference 
between $\Hom(G,H)$ and the space $\Map_{s\Top}(NG,N(H,V))$ of maps
between simplicial spaces is measured precisely by the continuous
2-cocycles on $G$ with values in $V$, modulo boundaries of 1-cocycles.  Because
$G$ is compact, Haar measure gives a contracting homotopy
\eqref{prop:deformation-retraction} on the
complex of continuous chains on $G$.    A sketch by the author of this
proof originally appeared  
as an answer to a question on the site
MathOverflow\footnote{``Equivariant 
  classifying spaces from classifying spaces'',
  \url{http://mathoverflow.net/q/223546}.}. 

Our approach gives a uniform proof of \eqref{thm:main} for all
1-truncated compact Lie groups $H$; furthermore, even in the case of
abelian $H$, it is somewhat more direct than the one given in
\cite{lashof-may-segal-equivariant-abelian}.  

We derived the unbased theorem \eqref{thm:unbased} from the based
version \eqref{thm:main} in \S\ref{sec:proof-unbased}, by comparing
associated fibrations over $BH$.  

The result on equivariant
classifying spaces \eqref{thm:equivariant} is proved in
\S\ref{sec:proof-equivariant}.  The proof relies on an explicit model,
built as the nerve of a certain topological category, of the
restriction of the universal $(G,H)$-bundle to the fixed point
subspace $(B_GH)^G\subseteq B_GH$.  The explicit model we use appears
to be essentially of the type  described in
\cite{guillou-may-merling-cat-models-eq-class-space}.  

In \S\ref{sec:space-of-homs}, we give
for the convenience of the reader a proof of the identification of
$\Hom(G,H)$ as mentioned in
\eqref{rem:unbased-in-terms-of-cent}.

\subsection{Acknowledgments}

I thank Peter May for comments on a draft of this paper.

\subsection{Conventions}

In this paper, we write $\Top$ for the category of \emph{compactly generated
weak Hausdorff spaces (CGWH)}, the standard convenient category of
spaces. 
This category is cartesian closed, and we 
write $\Map(X,Y)$ for the internal function object, i.e., continuous
maps with the k-ification of the compact-open topology.  We make use of
the ``usual'' model structure on $\Top$, in which weak equivalences
are weak equivalences on homotopy groups, and fibrations are Serre
fibrations.

\section{Nerve of a topological crossed module}
\label{sec:nerve-crossed-module}

\subsection{Crossed modules}\label{subsec:crossed-module}

Recall that a \dfn{crossed module} consists of
\begin{itemize}
\item groups $H$ and $V$,
\item a homomorphism $\epsilon\colon V\ra H$,
\item a homomorphism $\alpha\colon H\ra \Aut(V)$,
\item such that (i) $\epsilon(\alpha(h)(v))=h \epsilon(v) h^{-1}$
  and (ii) $\alpha(\epsilon(v))(v')=vv'v^{-1}$ for $h\in H$, $v,v'\in V$.
\end{itemize}
A \dfn{topological crossed module} is one in which $V$ and $H$ are
topological groups, and $\epsilon$ and $\alpha$ are continuous.  I'll
typically write $(H, V)$ for the crossed module, leaving $\epsilon$
and $\alpha$ understood.  Note that we will often
consider crossed modules in which $V$ is an abelian group, in which
case we will switch to additive notation for $V$, though not for $H$.

\begin{exam}
Given any group $H$, there is a unique crossed module which we will denote
$(H,0)$, in which $0$ is the trivial group.   
\end{exam}

\begin{exam}[Exponential crossed module]\label{exam:toral-crossed-module}
The following example is the crucial one for this paper.  Suppose $H$
is a 1-truncated compact Lie group.  We set
\begin{itemize}
\item $V:= T_eH$, the Lie algebra of $H$, which is a group under
  addition of vectors;
\item $\epsilon:= \exp\colon V\ra H$, the exponential map; this is a
  homomorphism since $H_0$ is abelian;
\item $\alpha := \ad \colon H\ra GL(V)$, the adjoint action. 
\end{itemize}
We typically write the group law of $V$ additively, so the identities
for the crossed module structure become
\[
\exp(\ad(h)(v)) = h\exp(v)h^{-1},\qquad \ad(\exp(v))(v') = v+v'-v=v'.
\]
The following features of this case will be significant:
\begin{enumerate}
\item 
$\alpha=\ad\colon H\ra GL(V)$ factors through the quotient group $H/H_0$,
\item 
$\epsilon=\exp\colon V\ra H$ is a covering map,
\item 
the underlying space of $V$ is contractible.
\end{enumerate}

\end{exam}

\subsection{Nerve of a crossed module}

The \dfn{nerve} of a topological crossed module $N(H,V)$ is the
simplicial space defined as follows; except for the topology, this is
as in \cite{brown-groupoids-crossed-objects}*{\S3.1}.  The space
$N(H,V)_n$ in     
degree $n$ is the space of tuples
\[
\bigl((h_{ij})_{0\leq i\leq j\leq n}, (v_{ijk})_{0\leq i\leq j\leq
  k\leq n}\bigr), \qquad h_{ij}\in H, \qquad v_{ijk}\in V,
\]
satisfying the identities
\begin{enumerate}
\item $h_{ii}=e$ and $v_{iij}=v_{ijj}=e$ for all $i\leq j$,
\item $h_{ik}=\epsilon(v_{ijk})h_{ij}h_{jk}$ for all $i\leq
  j\leq k$,
\item $v_{ik\ell}v_{ijk} = v_{ij\ell}\,\alpha(h_{ij})(v_{jk\ell}) $
  for all $i\leq j\leq k\leq \ell$.
\end{enumerate}
The action of simplicial operators $\delta\colon [n]\ra [m]$ is the
evident one: $(\delta 
h)_{ij}=h_{\delta(i),\delta(j)}$ and $(\delta
v)_{ijk}=v_{\delta(i),\delta(j),\delta(k)}$.  A standard 
argument shows that as a space $N(H,V)_n \approx H^n\times
V^{\binom{n}{2}}$, e.g., via the projection to coordinates $h_{0i}$,
$1\leq i\leq n$ and $v_{0ij}$, $1\leq i<j\leq n$.

Note that $N(H,V)_0=*$, i.e., $N(H,V)$ is a \emph{reduced}
simplicial space.

\begin{exam}
The nerve of
$N(H,0)$ is precisely the usual nerve of the group $H$; we write $N(H):=
N(H,0)$.  
\end{exam}

\subsection{Simplicial spaces and the Reedy model structure}

We write $s\Top$ for the category of simplicial spaces, i.e., functors
$\Delta^\op \ra \Top$.  We are going to use the \dfn{Reedy model
  structure} on $s\Top$.  We will need to use the following features
of this model structure:
\begin{enumerate}
\item Weak equivalences $f\colon X\ra Y$ in $s\Top$ are precisely the
  levelwise weak equivalences, i.e., $f_n\colon X_n\ra Y_n$ is a weak
  equivalence for all $n\geq0$.
\item An object $X$ is cofibrant (\dfn{Reedy cofibrant}) if and only the latching space
  inclusions $\gamma_n\colon L_nX\ra X_n$ are cofibrations in $\Top$.
\item An object $Y$ is fibrant (\dfn{Reedy fibrant}) if and only if
  the matching space 
  projections $\delta_n\colon Y_n\ra M_nY$ are fibrations in $\Top$.
\item The model structure is topological.  In particular, if $X$ is a 
  cofibrant simplicial space and $Y\ra Y'$ is a weak equivalence
  between fibrant simplicial spaces, then $\Map_{s\Top}(X,Y)\ra
  \Map_{s\Top}(X,Y')$ is a weak equivalence of spaces.
\end{enumerate}
We will need to examine latching and matching spaces in a bit more detail.

\subsection{Latching and matching spaces}

We recall the notion of latching and matching spaces.  For 
simplicial spaces $X\colon \Delta^\op\ra \Top$ and  all
$n\geq0$,  we have natural maps of spaces
\[
L_nX\xra{\gamma_n} X_n\xra{\delta_n} M_nX
\]
where 
\[
L_nX = \colim_{(\Delta^\op_{/[n]})_{<n}} X,\qquad M_n =
\llim_{(\Delta^\op_{[n]/})_{<n}} X,
\]
called the \dfn{latching} and \dfn{matching} spaces of $X$.  

\subsection{Latching spaces for the nerve of a group}

\begin{prop}\label{prop:latching-nerve-of-group}
Let $G$ be a topological group, and $NG\in s\Top$ its nerve.  Then for
each $n\geq0$, the latching inclusion $\gamma_n\colon L_n(NG)\ra
(NG)_n$ is isomorphic to the inclusion
\[
\set{(g_1,\dots,g_n)}{\exists i,\; g_i=e} \ra G^n.
\]
In particular, $NG$ is Reedy cofibrant if $\{e\}\ra G$ is a
cofibration in $\Top$.
\end{prop}
\begin{proof}
Standard.
\end{proof}

\subsection{Matching spaces for the nerve of a crossed module}

We describe the matching projections for the nerve of a topological
crossed module. 
\begin{prop}\label{prop:matching-projections}
Consider $N:=N(H,V)$ the nerve of a topological crossed module.  We write
$M_n:=M_nN$ for its matching spaces.  
\begin{enumerate}
\item [(0)] $\delta_0\colon N_0\ra M_0$ is the isomorphism of 1-point
  spaces.
\item [(1)] $\delta_1\colon N_1\ra M_1$ is the projection $H\ra *$.
\item [(2)] $M_2\approx H^{\times 3}$, and there is a pullback square
\[\xymatrix@C=100pt{
{N_2} \ar[d]_{\delta_2} \ar[r]
& {V} \ar[d]^{\epsilon}
\\
{M_2} \ar[r]_{(h_{01},h_{02},h_{12})\mapsto
  h_{02}h_{12}^{-1}h_{02}^{-1}} 
& {H}
}\]
\item [(3)] There is a commutative diagram
\[\xymatrix{
{N_3} \ar[r] \ar[d]_{\delta_3}
& {\{e\}} \ar[d]
\\
{M_3} \ar[r] \ar[d]
& {\Ker \epsilon} \ar[r] \ar[d]
& {\{e\}} \ar[d]
\\
{H^{\times 3}\times V^{\times 3}} \ar[r]_-{\pi}
& {V} \ar[r]_{\epsilon}
& {H}
}\]
in which all squares are pullback squares, and $\pi$ is given by 
\[
(h_{01},h_{12},h_{23}, v_{012}, v_{013}, v_{023}, v_{123}) \mapsto
v_{023}v_{012} \alpha(h_{01})(v_{123})^{-1} v_{013}^{-1}.
\]
\item [($\geq 4$)] $\delta_n\colon N_n\ra M_n$ is an isomorphism for
  $n\geq 4$.
\end{enumerate}
\end{prop}
\begin{proof}
Straightforward.  In (3), one shows directly that the right-hand lower
square, 
bottom rectangle, and left rectangle are pullbacks.
\end{proof}

Recall that a simplicial space $X$ is \dfn{Reedy fibrant} if each of
the maps $\delta_n\colon X_n \ra M_nX$ is a fibration of spaces.
\begin{cor}\label{cor:crossed-nerve-fib}
If $(H,V)$ is a topological crossed module such that $\epsilon$ is a
covering map, then $N(H,V)$ is Reedy fibrant.
\end{cor}
\begin{proof}
Immediate using \eqref{prop:matching-projections}.  Note that the
condition that $\epsilon$ be a covering map in
\eqref{prop:matching-projections}(3) implies that $\delta_3$ is an
open and closed embedding.
\end{proof}
In particular, \eqref{cor:crossed-nerve-fib} applies to our main example
\eqref{exam:toral-crossed-module}. 

\section{Maps between reduced simplicial spaces}
\label{sec:maps-simplicial-spaces}

A simplicial space $X\in s\Top$ is said to be \dfn{reduced} if
$X_0\approx *$.  We write $s\Top^\red\subset s\Top$ for the full
subcategory of reduced simplicial spaces.  Note that reduced simplicial
spaces are canonically based, so that we may in fact regard
$s\Top^\red$ as a full subcategory of simplicial based spaces $s\Top_*$.

\subsection{Realization of reduced simplicial spaces}

We recall the geometric realization functor $\Len{-}\colon s\Top\ra
\Top$, defined so that $\Len{X}$ is the coend of the functor $\Delta^\op\times \Delta\ra
\Top$ given by $([m],[n])\mapsto X_m\times \Delta^n$, where $\Delta^n$
is the topological $n$-simplex.

\begin{prop}\label{prop:realiz-adj}
The restriction of the geometric realization functor to a functor
$\Len{-}\colon s\Top^\red\ra \Top_*$ admits a right adjoint
$\nabla\colon \Top_*\ra s\Top^\red$, defined by 
\[
(\nabla Y)_n := \Map_*(\Delta^n/\Sk_0\Delta^n, Y),
\]
where $\Sk_0\Delta^n\subseteq \Delta^n$ is the set of vertices of the
simplex.  The adjunction is compatible with the topological
enrichment, and so gives a natural homeomorphism
\[
\Map_{s\Top}(X,\nabla Y) \approx \Map_*(\Len{X}, Y)
\]
for $X\in s\Top^\red$ and $Y\in \Top_*$.  
\end{prop}
\begin{proof}
This is a straightforward consequence of the observation that for
reduced simplicial spaces $X$, $\Len{X}$ is seen to be isomorphic to
the coend (in $\Top_*$) of $([m],[n]) \mapsto X_m\sm
(\Delta^n/\Sk_0\Delta^n)$.   
\end{proof}

We also note the following.
\begin{prop}\label{prop:sing-fib}
For any $Y\in \Top_*$, the simplicial space $\nabla Y$ is Reedy fibrant.
\end{prop}
\begin{proof}
The matching space projection has the form $\Map_*(\Delta^n/\Sk_0
\Delta^n,Y) \ra \Map_*(\partial \Delta^n/\Sk_0 \Delta^n,Y)$, which is
clearly a fibration.  
\end{proof}

For a topological group $H$, we consider the classifying space $BH:=
\Len{NH}$.
\begin{prop}\label{prop:unit-is-we}
If $H$ is a topological group with identity element a non-degenerate
basepoint (i.e., $\{e\}\ra H$ has the HEP), then the map
\[
\eta\colon NH \ra \nabla \Len{NH} = \nabla BH
\]
given by the unit map of the adjunction of \eqref{prop:realiz-adj} is
a levelwise weak equivalence of simplicial spaces.
\end{prop}
\begin{proof}
In degree $0$, $\eta$ is the isomorphism of one-point spaces.  In
degree $1$ it has the form 
\[
H \ra \Map_*(\Delta^1/\{0,1\}, \Len{NH}) \approx \Omega BH.
\]
A standard argument (e.g., using the usual simplicial model for the
universal fibration \cite{may-classifying-fibrations}) shows that this
is a weak equivalence.  

For $n\geq2$, we reduce to the $n=1$ case using the fact that
$I_n/\Sk_0\Delta^n\ra \Delta^n/\Sk_0\Delta^n$ is a homotopy
equivalence of pointed spaces, and thus
\[
\Map_*(\Delta^n/\Sk_0\Delta^n, \Len{NH}) \ra \Map_*(I_n/\Sk_0\Delta^n,
\Len{NH}) \approx (\Omega BG)^{\times n}
\]
is a weak equivalence, where $I_n\subseteq \Delta^n$ is the union of
the edges with vertices $\{k-1,k\}$ for all $k=1,\dots,n$.
\end{proof}

\subsection{$\Map(X, N(H,V))$ computes $\Map_*(\Len{X}, BH)$}

Now we fix a 1-truncated compact Lie group $H$ and the corresponding
exponential  crossed module
$(H,V)$ of \eqref{exam:toral-crossed-module}. 
We have a map of reduced simplicial spaces
\[
NH \xra{(\iota,\eta)} N(H,V)\times \nabla \Len{NH}
\]
in which $\iota$ is the evident inclusion $NH=N(H,0)\subseteq N(H,V)$,
and $\eta$ the unit map of the adjunction of
\eqref{prop:realiz-adj}.   Observe that both $\iota$ and $\eta$ are
levelwise weak equivalences ($\iota$ because $V$ is contractible, $\eta$ by
\eqref{prop:unit-is-we}).  Furthermore,  both $N(H,V)$
\eqref{cor:crossed-nerve-fib}   and $\nabla
\Len{NH}$ \eqref{prop:sing-fib} are Reedy fibrant.

Using the Reedy model structure on
simplicial spaces, we can factor the above map as 
\begin{equation}\label{eq:Reedy-factorization}
NH \xra{j} (NH)^f \xra{(\iota',\eta')} N(H,V)\times \nabla \Len{NH}
\end{equation}
so that $(NH)^f$ is Reedy fibrant and $j$ is a levelwise weak
equivalence, whence $\iota'$ and $\eta'$ are also levelwise weak
equivalences.  

\begin{prop}\label{prop:mapping-space-equiv}
For $X$ a Reedy cofibrant simplicial space with $X_0=*$, and $(H,V)$
the exponential crossed module of a 1-truncated compact Lie group $H$,
we have that  
$\Map_{s\Top}(X,N(H,V))$ is weakly equivalent to $\Map_*(\Len{X}, 
 \Len{NH})$.  Furthermore, $\iota_*\colon \Map_{s\Top}(X,NH)\ra
\Map_{s\Top}(X,N(H,V))$ is a weak equivalence of spaces if and only if
$\eta_*\colon \Map_{s\Top}(X,NH)\ra \Map_{s\Top}(X,\nabla \Len{NH})$
is. 
\end{prop}
\begin{proof}
Straightforward using the factorization
\eqref{eq:Reedy-factorization}, the fact that
Reedy model structure is compatible with the topological enrichment,
and the adjunction \eqref{prop:realiz-adj}. 
\end{proof}

\begin{cor}\label{cor:reduction-to-nerve}
If $(H,V)$ is as above, and  $G$ is a topological group such that
$\{e\}\ra G$ is a cofibration, 
then  $\Map_*(BG,BH)$ is weakly equivalent to $\Map_{s\Top}(NG,
N(H,V))$, and 
\[
B\colon \Hom(G,H) \ra \Map_*(BG,BH)
\]
is a weak equivalence if and only if 
\[
\iota_*\colon \Map_{s\Top}(NG,NH)\ra \Map_{s\Top}(NG, N(H,V))
\]
is a weak equivalence.
\end{cor}
\begin{proof}
Use \eqref{prop:mapping-space-equiv} with $X=NG$, which is Reedy
cofibrant by \eqref{prop:latching-nerve-of-group}. 
It is straightforward to see that $\Map_{s\Top}(NG,NH)\ra \Hom(G,H)$
(evaluation at spaces in degree 1) 
is a homeomorphism, and so the map $B$ coincides with $\iota_*$.
\end{proof}

\begin{rem}
If $H$ is a discrete group, then $NH$ is already Reedy fibrant, in
which case we can 
immediately derive the well-known fact that $B\colon \Hom(G,H)\ra
\Map_*(BG,BH)$ is a weak equivalence for any such topological group $G$.
\end{rem}

\section{Proof of \eqref{thm:main}: based mapping space}
\label{sec:proof-based}

As above, we assume that $H$ is a 1-truncated compact Lie group.  We
will now also assume 
that $G$ is a compact Lie group.  By \eqref{cor:reduction-to-nerve},
we have reduced \eqref{thm:main} to showing that
$\Map_{s\Top}(NG,NH)\ra \Map_{s\Top}(NG, N(H,V))$ is a weak equivalence.

Let $E:=\Map_{s\Top}(NG, N(H,V))$.  Using
\eqref{prop:matching-projections} and the identification of the
latching inclusions $L_nNG\ra G^n$
\eqref{prop:latching-nerve-of-group}, we see that $E$ is precisely the
space of pairs 
\[
(\zeta,\nu)\in \Map(G,H)\times \Map(G\times G, V)
\]
such that
\begin{enumerate}
\item $\zeta(e)=e$ and $\nu(g,e)=0=\nu(e,g)$ for $g\in G$,
\item $\zeta(g_1g_2) = \exp[\nu(g_1,g_2)]\zeta(g_1)\zeta(g_2)$ for
  $g_1,g_2\in G$,
\item $\nu(g_1g_2,g_3)+\nu(g_1,g_2) =
  \nu(g_1,g_2g_3)+\ad(\zeta(g_1))[\nu(g_2,g_3)]$ for $g_1,g_2,g_3\in G$.
\end{enumerate}
Explicitly, this corresponds to the map $NG\ra N(H,V)$ which (in the
notation of \S\ref{subsec:crossed-module}) sends
$(g_{ij})\in (NG)_n$ to $(h_{ij}, v_{ijk})\in N(H,V)_n$ with
$h_{ij}=\zeta(g_{ij})$ and $v_{ijk} = \nu(g_{ij},g_{jk})$.  

Let $E^0:=\Map_{s\Top}(NG,NH)$.  The map $E^0\ra E$ is precisely
inclusion into the subspace consisting of points of the form
$(\zeta,0)$.  

For a continuous map $\zeta\colon G\ra H$, we write $\ol\zeta\colon
G\ra H/H_0$ for the composite with the quotient map $H\ra H/H_0$.  
Note that if $(\zeta,\nu)\in E$, then  $\ol\zeta$ 
is a 
continuous homomorphism of groups.  Since $H/H_0$ is discrete, we obtain
coproduct decompositions
\[
E=\coprod_\gamma E_\gamma, \qquad
E^0=\coprod_\gamma E^0_\gamma, \qquad \gamma\in \Hom(G,H/H_0).
\]
Thus, we must show that for each such $\gamma$, the inclusion
$E^0_\gamma\subseteq E_\gamma$ is a weak equivalence.  In fact, we can give
an explicit (strong) deformation retraction of $E_\gamma$ to 
$E^0_\gamma$, which  relies on the existence of a contracting
homotopy of the complex $C^*(G,V_{\ad\gamma})$ of normalized continuous
cochains on $G$ with 
values in the representation $\ad\gamma\colon G\ra \Aut(V)$, which may  be
constructed explicitly  using an invariant measure on the compact group $G$.
We spell out the details we need below.

Fix $\gamma\in \Hom(G,H/H_0)$.  Let $C^1_\gamma\subseteq \Map(G,V)$ be
the subspace of functions $\mu\colon G\ra V$ such that 
\[
\mu(e)=0.
\]
Let $Z^2_\gamma\subseteq \Map(G\times G,V)$ be the subspace of
functions $\nu\colon G\times G\ra V$ such that
\[
\nu(g,e)=0=\nu(e,g),\qquad g\in G,
\]
and 
\[
\nu(g_1g_2,g_3) + \nu(g_1,g_2) = \nu(g_1,g_2g_3) +
\ad\gamma(g_1)[\nu(g_2,g_3)],\qquad g_1,g_2,g_3\in G.
\]
Both $Z^2_\gamma$ and $C^1_\gamma$ are topological real vector spaces.  
Define continuous and linear maps
\[
d\colon C^1_\gamma\ra Z^2_\gamma,\qquad H\colon Z^2_\gamma\ra C^1_\gamma
\]
by 
\begin{align*}
d\mu(g_1,g_2) &:= \mu(g_1) - \mu(g_1g_2) + \ad\gamma(g_1)\mu(g_2), 
\\
H\nu(g) &:= \int_G x^{-1} \nu(x,g)\,dx,
\end{align*}
where we use right-invariant Haar measure
 on $G$ normalized so that $\int_Gdx=1$.
\begin{lemma}\label{lemma:contracting-homotopy}
The composite $dH\colon Z^2_\gamma\ra Z^2_\gamma$ is the identity
map. 
\end{lemma}
\begin{proof}
For $g\in G$ and $v\in V$ we write ``$gv$'' for $\ad\gamma(g)(v)$
below.  Given $\nu\in Z^2_\gamma$ we have
\begin{align*}
  dH\nu(g_1,g_2) & = \int_G x^{-1}\nu(x,g_1) - x^{-1}\nu(x,g_1g_2) +
                   g_1x^{-1}\nu(x,g_2)\,dx
\\
&= \int_G x^{-1}\nu(x,g_1)
  -x^{-1}[\nu(xg_1,g_2)+\nu(x,g_1)-x\nu(g_1,g_2)]+g_1x^{-1}\nu(x,g_2)\,dx
\\
&= \nu(g_1,g_2) - \int_G g_1(xg_1)^{-1}\nu(xg_1,g_2)\,dx + \int_G
  g_1x^{-1}\nu(x,g_2)\,dx = \nu(g_1,g_2),
\end{align*}
the last cancellation by right-invariance of the measure.
\end{proof}

\begin{prop}\label{prop:deformation-retraction}
The inclusion $E^0_\gamma\subseteq E_\gamma$ admits a strong
deformation retraction.
\end{prop}
\begin{proof}
Define $K_t\colon E_\gamma\ra E_\gamma$ for $0\leq t\leq 1$ by
$K_t(\zeta,\nu):= (\zeta_t,\nu_t)$, with
\begin{align*}
  \zeta_t(g) &:= \exp[ t\,H\nu(g)] \zeta(g),
\\
  \nu_t(g_1,g_2) &:= \nu(g_1,g_2) - t\, dH\nu(g_1,g_2).
\end{align*}
We have $K_0=\id_{E_\gamma}$, $K_t|E^0_\gamma=\id_{E^0_\gamma}$, and
$K_1(E_\gamma)\subseteq E^0_\gamma$ as desired, the last using
\eqref{lemma:contracting-homotopy}. 
\end{proof}

The proof of \eqref{thm:main} follows, using
\eqref{prop:mapping-space-equiv} and the remarks above.

\begin{rem}
If $H$ is an abelian group, then $\ad\colon H\ra \Aut(V)$ is trivial.
In this case, the proof of \eqref{prop:deformation-retraction}
directly gives a deformation retraction of $E^0\subseteq E$.
\end{rem}

\section{Proof of \eqref{thm:unbased}: Unbased mapping space} 
\label{sec:proof-unbased}

Given simplicial spaces $X$ and $Y$, we have an internal function object
$\eMap(X,Y)\in s\Top$, characterized so that $\eMap(X,-)$ is the right
adjoint to $(-)\times X$.  We have that
\[
\eMap(X,Y)_n = \Map_{s\Top}(X\times N[n], Y),
\]
where $[n]$ is the $n$-arrow category.  In particular,
$\eMap(X,Y)_0\approx \Map_{s\Top}(X,Y)$.  

Formation of the internal function object is compatible with
realization: there are canonical maps
\begin{equation}\label{eq:reaz-function}
\rho \colon \Len{\eMap(X,Y)} \ra \Map(\Len{X},\Len{Y})
\end{equation}
natural in $X$ and $Y$.  This map exists exactly because the
realization functor $\Len{-}\colon s\Top\ra \Top$ preserves finite
products, and is characterized as the map adjoint to
\[
\Len{\eMap(X,Y)}\times \Len{X} \xla{\sim} \Len{\eMap(X,Y)\times X}
\xra{\len{\mathrm{eval}}} \Len{Y}.  
\]

Given topological groups $G$ and $H$, we consider the function object
$\eMap(NG,NH)$.  We have an evident isomorphism
\[
\eMap(NG,NH) \approx N\Fun(G,H), 
\]
where $\Fun(G,H)$ is the internal category in $\Top$ of functors and
natural transformations from $G$ to $H$.  Explicitly, this has
\begin{itemize}
\item objects $\phi\in \Hom(G,H)$,  and
\item morphisms $\phi_0 \xra{h} \phi_1$ where $h\in H$,
  $\phi_1=h\phi_0h^{-1}$. ,
\end{itemize}
and thus homeomorphisms
$N\Fun(G,H)_n=\eMap(NG,NH)_n = \Hom(G,H)\times H^{\times n}$.  

Write $(H\curvearrowright H)$ for the translation category of the left
action of $H$ on itself, viewed as a category object in $\Top$.  This
has 
\begin{itemize}
\item objects $h_0\in H$, 
\item morphisms $h_0\xra{h} h_1$ where $h\in H$, $h_1=hh_0$.
\end{itemize}
We have homeomorphisms $N(H\curvearrowright H)_n= H^{\times (n+1)}$.
The group $H$ acts on the category $(H\curvearrowright H)$ by
$\delta\cdot h_0=h_0\delta^{-1}$ (on objects) and $\delta\cdot
(h_0\xra{h} h_1) = h_0\delta^{-1}\xra{h} h_1\delta^{-1}$ (on
morphisms), where $\delta\in H$.  
 
We let 
$EH:= \Len{N(H\curvearrowright H)}$, a contractible $H$ space with
free $H$-action.

\begin{lemma}\label{lemma:realiz-of-functor-cat}
There is a homeomorphism $\Len{\eMap(NG,NH)}\approx (\Hom(G,H)\times
EH)/H$, where $H$ acts on $\Hom(G,H)$ by conjugation.
\end{lemma}

\begin{proof}[Proof of \eqref{thm:unbased}]
We have a commutative diagram
\[\xymatrix{
{\Len{\eMap(NG,NH)}} \ar[r]^-{\rho} \ar[d]_{\alpha}
& {\Map(\Len{NG},\Len{NH})} \ar[d]^{\beta}
\\
{\Len{\eMap(*,NH)}} \ar[r]_-{\approx}
& {\Map(\Len{*}, \Len{NH})}
}\]
where the vertical maps are induced by restriction along $*\ra NG$,
and the lower horizontal map is the evident homeomorphism (both source
and target are homeomorphic to $BH$). 
We claim that 
$\rho$ is a weak equivalence.

By
\eqref{lemma:realiz-of-functor-cat} we see that $\alpha\colon
(\Hom(G,H)\times EH)/H \ra BH$ is a fiber bundle with fiber $\Hom(G,H)$.
Since $\beta\colon 
\Map(BG,BH)\ra BH$ is also a fibration, and the base space $BH$ is
path connected, $\rho$ is a weak equivalence if and only if its
restriction to the fiber over the base point is, which is precisely
the weak equivalence $\Hom(G,H)\ra \Map_*(BG,BH)$ of
\eqref{thm:main}.  
\end{proof}

\section{Proof of \eqref{thm:equivariant}: Equivariant classifying
  space}
\label{sec:proof-equivariant}

\subsection{Recollections on equivariant bundles}

A \dfn{$G$-equivariant principal $H$ bundle} (or \dfn{$(G,H)$-bundle}),
is a principal $H$-bundle $\pi\colon P\ra X$, together with actions of
$G$ on $P$ and $X$, compatible with $\pi$, so that $G$ acts via maps
of principal $H$-bundles.  We will always assume that both $G$ and $H$
are compact Lie groups.

This definition is somewhat anomalous, in that $(G,H)$-bundles are not
characterized by a property which is local in $X$.  Thus, we say that
a $(G,H)$-bundle is \dfn{locally trivial} if it looks locally like
\[
(G\times H)\times_{\Lambda_\phi} U \ra G\times_{G'} U,
\]
where $G'\leq G$ is a closed subgroup, $\Lambda_\phi:=
\set{(g,\phi(g))}{g\in G'}$ is the graph of some homomorphism
$\phi\colon G'\ra H$, and $\Lambda\xra{\sim} G'$ acts on a space $U$. 
The key result is that if $G$ and $H$ are compact and $X$ is
completely regular\footnote{Completely regular = points are closed,
  and any point and disjoint closed subset are separated by a real
  valued function.}, then any $(G,H)$-bundle over $X$ is locally
trivial \cite{lashof-equivariant-bundles}*{Cor.\ 1.5}.  

A $(G,H)$-bundle $P\ra X$ is \dfn{numerable} it admits a locally
trivializing cover which itself admits a subordinate partition of
unity by $G$-invariant functions.  Over a paracompact base $X$, every
locally trivial bundle is numerable
\cite{lashof-equivariant-bundles}*{Cor.\ 1.13}.  There is a universal
$(G,H)$-bundle $E_GH\ra B_GH$, which classifies equivalence classes of
numerable 
bundles: see \cite{lashof-equivariant-bundles},
\cite{lashof-may-gen-equivariant-bundles}, and also
\cite{luck-uribe-equivariant-principal-bundles} for a recent and more
general treatment. 

We will be mainly concerned with the case of $(G,H)$-bundles
$\pi\colon P\ra X$ such that $G$-acts trivially on $X$.  In such a
case there is a natural function
\[
\tau\colon P\ra \Hom(G,H)
\]
defined so that $\tau(p)(\gamma)\in H$ is the unique $\delta\in H$
such that $(\gamma,\delta)\cdot p=p$.  When $P\ra X$ is locally
trivial, the map $\tau$ is seen to be continuous.  Observe that $\tau$
is $G\times H$-equivariant, where this group acts on $\Hom(G,H)$ by
conjugation: $(\gamma,\delta)\cdot \phi = \delta\phi(\gamma)^{-1}\phi
\phi(\gamma)\delta^{-1}$.  
\begin{lemma}\label{lemma:tau-is-fibration}
For any locally trivial $(G,H)$-bundle $\pi\colon P\ra X$ over a
$G$-fixed base $X$, the map
$\tau\colon P\ra \Hom(G,H)$ is a Serre fibration.
\end{lemma}
\begin{proof}
This will follow by showing that $(\tau,\pi)\colon P\ra
\Hom(G,H)\times X$ is actually fiber bundle.  Since $\pi$ is locally
trivial, we can reduce to the case when $\pi$ has the form $\pi\colon
(G\times H)/\Lambda_\phi \times U\ra U$, where $\Lambda_\phi\leq
G\times H$ is the graph of some homomorphism $\phi\colon G\ra H$.
Then 
\[
(\tau,\pi)=\rho\times \id_U \colon (G\times H)/\Lambda_\phi \times U
\ra 
\Hom(G,H)\times U, 
\]
where $\rho\colon (G\times H)/\Lambda_\phi\ra \Hom(G,H)$ sends
$[\gamma,\delta]\mapsto
\delta\phi(\gamma)^{-1}\phi(\delta\phi(\gamma)^{-1})^{-1}$.  Because $\Hom(G,H)$
is topologically a coproduct of orbits under $H$-conjugation
\eqref{rem:unbased-in-terms-of-cent}, \eqref{prop:hom-space-topology},
we see that $\rho$ is isomorphic 
to the composite of a projection map $(G\times H)/\Lambda_\phi \ra
H/C_H(\phi)$ (induced by $(\gamma,\delta)\mapsto
\delta\phi(\gamma)^{-1}$) with an open and closed immersion, and thus
is a fibration. 
\end{proof}

\subsection{Outline of the proof}

To prove that the map $B_GH\ra \Map(EG,B_GH)$ (induced by restriction
along $EG\ra *$) is a $G$-equivariant weak
equivalence, it suffices to show that it induces a weak equivalence of spaces
$(B_GH)^{G'}\ra \Map(EG,B_GH)^{G'}$ for all closed subgroups $G'\leq
G$.  Without loss of generality, we may assume $G'=G$, since, when the
group action is restricted to  the subgroup $G'$, $B_GH$ is a
$B_{G'}H$ and $EG$ is an $EG'$.  Thus, we will show that $(B_GH)^G\ra
\Map(EG,B_GH)^G$ is an equivalence, using the following.

\begin{lemma}\label{lemma:htpy-commutative-classifying}
Suppose given a $(G,H)$-bundle $P\ra X$ over a
space $X$ with trivial $G$-action, together with maps:
\begin{enumerate}
\item
$\alpha\colon X\ra (B_GH)^G\subseteq B_GH$ classifying the
$G$-equivariant $H$-bundle $P\ra X$, i.e.,
covered by a $(G,H)$-bundle map $P\ra E_GH$, and
\item $\rho\colon X\ra \Map(BG,BH)$, whose adjoint $\wt\rho\colon X\times
  BG=(X\times EG)/G\ra BH$ classifies the $H$-bundle $(P\times EG)/G\ra
  (X\times EG)/G$, i.e., covered by an $H$-bundle map $(P\times EG)/G \ra EH$.
\end{enumerate}
Then the diagram
\begin{equation}\label{eq:htpy-commutative-square}
\vcenter{
\xymatrix{
{X} \ar[r]^{\alpha} \ar[d]_{\rho}
& {(B_GH)^G} \ar[d]
\\
{\Map(BG,BH)} \ar[r]_-{\sim}
& {\Map(EG,B_GH)^G}
}}
\end{equation}
commutes up to homotopy, where the bottom map is induced by a
$G$-equivariant map $\upsilon \colon BH\ra B_GH$ (with $G$ acting
trivially on $BH$) which classifies the universal $H$-bundle viewed as
a $G$-equivariant $H$-bundle with trivial $G$-action.  

Furthermore, the bottom map of the diagram is a weak equivalence.
\end{lemma}
\begin{proof}
The adjoints of both composite maps $X\ra \Map(EG,B_GH)^G$ are
$G$-equivariant maps $(X\times EG)/G\ra B_GH$, which in either case
are covered by maps $(P\times EG)/G\ra E_GH$ of $(G,H)$-bundles.  The
homotopy-commutativity of the diagram follows from the universal
property  
of $B_GH$, as the classifying space for such bundles.

To see that the bottom map of the diagram is a weak equivalence, note
that it may be constructed as follows.  Choose any $G\times H$
equivariant map $EH\ra E_GH$ (unique up to homotopy by the defining
property of $E_GH$), and take quotient with respect to the free
$H$-actions, obtaining a $G$-equivariant map $\upsilon \colon BH\ra
B_GH$, which classifies the 
universal bundle as stated.  By construction, $\upsilon$ is a
$G$-equivariant map which is a weak
equivalence on underlying spaces.

For any $G$-space $X$ and subgroup $G'\leq G$, the fixed point space
$\Map(EG,X)^{G'}$ is the space of ordinary homotopy fixed points of
the $G'$-action on $X$; as a consequence, $\Map(EG,-)$  takes a map of
$G$-spaces which is a  weak equivalence of underlying spaces, to a
$G$-equivariant weak equivalence.  In particular, we see that
$\Map(EG, BH) \ra \Map(EG,B_GH)$ is a
$G$-equivariant weak equivalence.  Taking $G$-fixed points gives an
equivalence 
$\Map(BG,BH)\approx \Map(EG,BH)^G\ra \Map(EG,B_GH)^G$, which is the map
in the diagram.

\end{proof}

The strategy is as follows.  Fix compact Lie groups $G$ and $H$, and  take 
  $\rho\colon X\ra \Map(BG,BH)$ in
  \eqref{lemma:htpy-commutative-classifying} to be  
 isomorphic to the 
map $(\Hom(G,H)\times EH)/H\ra \Map(BG,BH)$ described in
\S\ref{sec:proof-unbased}, which for 1-truncated $H$
gives the weak 
equivalence of 
\eqref{thm:unbased}.   We will
\begin{enumerate}
\item  construct a certain $(G,H)$-bundle $P\ra X$
(where $G$ acts trivially on $X$), 
\item prove that a map $\alpha\colon X\ra B_GH$ classifying $P\ra X$
  induces a weak equivalence $X\xra{\sim} (B_GH)^G\subseteq B_GH$, and
\item construct a bundle map $(P\times EG)/G\ra EH$ covering the
  adjoint $\wt\rho\colon X\times BG\ra BH$ to $\rho$.
\end{enumerate}
Thus by \eqref{lemma:htpy-commutative-classifying} both $\alpha$ and
$\rho$ fit in  a homotopy
commutative square \eqref{eq:htpy-commutative-square}.  It follows
that if $H$ is 1-truncated, \eqref{thm:unbased} implies that $\rho$ is
a weak equivalence, from which it follows that the right-hand vertical
arrow is a weak equivalence, 
which is the desired  result.  Note: the hypothesis that $H$ is
1-truncated is used only to show 
that $\rho$ (which exists for arbitrary $H$) is a weak equivalence. 

\subsection{Step 1: Construction of $P\ra X$}

As in the previous section, we consider categories
$\Fun(G,H)$ and $(H\curvearrowright H)$ (internal to $\Top$), where $G$
and $H$ are compact 
Lie groups.  Consider the topological
category $C$ defined as the fiber product
\[
C:= \Fun(G,H)\times_H (H\curvearrowright H)
\]
via the evident restriction functors  $\Fun(G,H)\ra \Fun(\{e\},H)=H$
and $(H\curvearrowright H)\ra 
(H\curvearrowright *)=H$.  (Here $H$ represents a topological category
with one object.)

The group $G\times H$ acts on $C$ via
\[
(\gamma,\delta)\cdot (\phi_0,h_0) = (\phi_0,
\phi_0(\gamma)h_0\delta^{-1}), 
\]
on objects and
\[
(\gamma,\delta)\cdot (\phi_0\xra{h}\phi_1, h_0\xra{h}h_1) =
(\phi_0\xra{h}\phi_1, \phi_0(\gamma)h_0\delta^{-1}\xra{h} \phi_1(\gamma)h_1\delta^{-1})
\]
on morphisms, $(\gamma,\delta)\in G\times H$.
(This works exactly because
$h\phi_0(\gamma)=\phi_1(\gamma) h$.)  The evident
projection functor $C\ra \Fun(G,H)$ is invariant under the $G\times H$ action
on $C$, with respect to the trivial action on $\Fun(G,H)$.  

We set $P:= \Len{NC}$ and $X:=\Len{N\Fun(G,H)}$, with $P\ra X$ induced
by the evident projection functor.  It is straightforward to show that
the induced $G\times H$-action on $P$ is compatible with the
projection map to $X$, and that $H$ acts freely on $P$ with
$P/H\approx X$.  In particular, $P\ra X$ has the structure of a
$G$-equivariant principal $H$-bundle.  

We note an equivalent description of $C$, and hence of $P$.  
Let $C':=\Hom(G,H)\times N(H\curvearrowright H)$, where $\Hom(G,H)$ is
viewed as a topological category with only identity maps.  There is an
isomorphism $C'\ra C$ of topological categories, given on objects and
morphisms by 
\[
(\phi,h_0)\mapsto (h_0\phi h_0^{-1}, h_0),\qquad 
(\phi, h_0\xra{h} h_1) \mapsto (h_0\phi h_0^{-1}\xra{h} h_1\phi
h_1^{-1}, h_0\xra{h} h_1).
\]
The $G\times H$-action on $C'$ induced by this isomorphism is
described by
\[
(\gamma,\delta)\cdot (\phi,h_0) = (\delta\phi(\gamma)^{-1}
\phi\phi(\gamma)\delta^{-1}, 
h_0\phi(\gamma)\delta^{-1}), 
\]
\[
(\gamma,\delta)\cdot (\phi,h_0\xra{h}h_1) =
(\delta\phi(\gamma)^{-1}\phi\phi(\gamma)\delta^{-1},
h_0\phi(\gamma)\delta^{-1}\xra{h} h_1\phi(\gamma)\delta^{-1}).  
\]
In particular, the projection functor $C'\ra \Hom(G,H)$ induces a
$G\times H$-equivariant map $P\ra \Hom(G,H)$ (using the conjugation
$G\times H$ action on $\Hom(G,H)$), and this map is a non-equivariant
weak equivalence, since $P\approx \Len{NC'}\approx \Hom(G,H)\times EH$.  

\subsection{Step 2: The weak equivalence $\alpha\colon X\ra
  (B_GH)^G$} 

Choose any $X\ra B_GH$ classifying the $P\ra X$ constructed above
(this exists because $X$ is paracompact and completely regular, so
numerable), and so covered by a $G\times H$-equivariant  map $P\ra E_GH$.    Since the
action of $G$ on $X$ is trivial, these factor through $\alpha\colon
X\ra (B_GH)^G$ and $\alpha'\colon P\ra p^{-1}((B_GH)^G)$, where
$p\colon E_GH\ra B_GH$ is the universal bundle.

\begin{lemma}
The map $\alpha'\colon P\ra p^{-1}((B_GH)^G)$ is a weak equivalence of
underlying spaces.
\end{lemma}
\begin{proof}
The map $\alpha'$ fits in a commutative diagram,
\[\xymatrix{
{P} \ar[rr]^{\alpha'}  \ar[dr]_{\tau}
&& {p^{-1}((B_GH)^G)} \ar[dl]^{\tau}
\\
& {\Hom(G,H)}
}\]
where by \eqref{lemma:tau-is-fibration} both maps marked $\tau$ are
Serre fibrations.  The fibers of these $\tau$s over $\phi\in
\Hom(G,H)$ are $EH$ and 
$(E_GH)^{\Lambda_\phi}$ respectively, both of which spaces are
contractible.  Thus  $\alpha'$ is a weak equivalence (as are both $\tau$s).
\end{proof}

It follows that $\alpha\colon X\ra (B_GH)^G$ is a weak equivalence, as
it is obtained by the quotient of $\alpha'$ by free $H$-actions.

\subsection{Step 3: The bundle map covering $\rho\colon X\ra \Map(BG,BH)$}  

We have a commutative square of functors
\[\xymatrix{
{C\times (G\curvearrowright G) = \Fun(G,H)\times_H (H\curvearrowright
  H)\times (G\curvearrowright G)} \ar[r] \ar[d] 
& {(H\curvearrowright H)} \ar[d]
\\
{\Fun(G,H)\times (G\curvearrowright G)} \ar[r]
& {H}
}\]
where the vertical arrows are the evident projections, the top
horizontal arrow is given by 
\[
(\phi_0,h_0,g_0) \mapsto (\phi_0(g_0)h_0),
\qquad
(\phi_0\xra{h} \phi_1, h_0\xra{h}h_1, g_0\xra{g} g_1) \mapsto 
(\phi_0(g_0)h_0 \xra{h\phi_0(g)=\phi_1(g)h} \phi_1(g_1)h_1). 
\]
on objects and morphisms, and the bottom horizontal arrow is given by 
\[
(\phi_0,g_0) \mapsto *, \qquad (\phi_0\xra{h}\phi_1, g_0\xra{g}
g_1)\mapsto (*\xra{h\phi_0(g)=\phi_1(g)h} *). 
\]
The group $G$ acts on the objects on the left-hand side of the square,
where $G$ acts on $C$ as described above, by the tautological right
action on $(G\curvearrowright G)$, and trivially on $\Fun(G,H)$. 
The horizontal arrows are
invariant under this $G$ action.  

Thus, taking geometric realizations of nerves and passing to quotients
by $G$-actions, we obtain a commutative square
\[\xymatrix{
{(P\times EG)/G} \ar[r] \ar[d]
& {EH} \ar[d]
\\
{(X\times EG)/G} \ar[r]
& {BH}
}\]
which is evidently a map of $H$-bundles.  Under the identification
$(X\times EG)/G \approx X\times BG \approx N(\Fun(G,H)\times G)$, we
see that the bottom arrow is isomorphic to that obtained from the
evaluation functor $\Fun(G,H)\times G\ra H$, and thus is adjoint to
the map $\rho\colon X\ra \Map(BG,BH)$ described earlier.

\section{The space of homomorphisms between compact Lie groups}
\label{sec:space-of-homs}

Recall that $\Hom(G,H)$ denotes the space of homomorphisms equipped
with the compact-open topology.  We give a proof of the following
fact, which is standard but not easily read from the literature with
which I am 
familiar.
\begin{prop}\label{prop:hom-space-topology}
Let $G$ and $H$ be Lie groups, with $G$ compact.
The  map
\[
(\phi,hC(\phi))\mapsto h\phi h^{-1}\colon \coprod_{[\phi]}
H/C(\phi)\ra \Hom(G,H), 
\]
where $C(\phi)=\set{h\in H}{\phi(g)h=h\phi(g)\;\forall
  g\in G}$, and $[\phi]$ runs over a set of $H$-conjugacy classes in 
$\Hom(G,H)$, is a homeomorphism.  In particular, $\Hom(G,H)$ is
locally compact, and thus a CGWH space.
\end{prop}
\begin{proof}
We quote a classical theorem of
Montgomery-Zippin
\cite{montgomery-zippin-topological-transformation-groups}*{p.\ 
  216}: for every compact subgroup $K$ of a Lie group $L$, there
exists a neighborhood $U$ of $K$ such that every closed subgroup of
$L$ in $U$ is $L$-conjugate to a subgroup of $K$.  Applied to
$L=G\times H$, and
$K=\Lambda_\phi=\set{(g,\phi(g)}{g\in G}$, the graph of
a continuous homomorphism $\phi\colon G\ra H$, we obtain a
neighborhood $U\subseteq G\times H$ of $\Lambda_\phi$ such that if
$\Lambda_{\phi'}\in U$ for $\phi'\in \Hom(G,H)$, then
$\phi'$ is $H$-conjugate $\phi$
\cite{conner-floyd-diff-per-maps}*{38.1}.  

There exists a neighborhood $V$ of $e\in H$ such that
$\Lambda_{\phi}\subseteq \set{(g,h)}{h\phi(g)^{-1}\in V}\subseteq U$.
To see this, use the homeomorphism $\alpha\colon G\times H\ra G\times
H$, $\alpha(g,h)=(g,h\phi(g)^{-1})$, together with the tube lemma
applied to $G\times \{e\}\subseteq \alpha(U)$.  

By definition, the set $V':= \set{f\colon G\ra H}{f(G)\subseteq V}$ is
an open subset of $C(G,H)$, the space of continuous maps $G\ra H$
equipped with the compact-open topology.  The space $C(G,H)$ is a
topological group under pointwise multiplication in $H$; to prove
this, use the fact that $G$, $H$, and finite products thereof are
locally compact, so that the relevant evaluation maps are continuous.
Therefore, the translated subset $V'\phi$ is open in $C(G,H)$.
Tracing through the definitions, we see that any continuous
homomorphism $G\ra H$  in $V'\phi$ must be conjugate to $\phi$.

Thus, we have shown that conjugacy classes are open subsets of
$\Hom(G,H)$.  

Now consider the map of the proposition.  Each $H/C(\phi)$ maps
bijectively to a conjugacy class in $\Hom(G,H)$.  As $H$ is Hausdorff,
so is $C(G,H)$ and hence so is the subspace $\Hom(G,H)$.  Therefore, each
$H/C(\phi)\ra 
\Hom(G,H)$ gives  a homeomorphism to its image, since
$H/C(\phi)$ is 
compact.  Because the image is also open, the homeomorphism of the
proposition follows.  

As an immediate consequence, we see that $\Hom(G,H)$ is a coproduct of
compact Hausdorff spaces, and thus locally compact.
\end{proof}



\begin{bibdiv}
\begin{biblist}
\bib{brown-groupoids-crossed-objects}{article}{
  author={Brown, Ronald},
  title={Groupoids and crossed objects in algebraic topology},
  journal={Homology Homotopy Appl.},
  volume={1},
  date={1999},
  pages={1--78 (electronic)},
  issn={1512-0139},
}

\bib{conner-floyd-diff-per-maps}{book}{
  author={Conner, P. E.},
  author={Floyd, E. E.},
  title={Differentiable periodic maps},
  series={Ergebnisse der Mathematik und ihrer Grenzgebiete, N. F., Band 33},
  publisher={Academic Press Inc.},
  place={Publishers, New York},
  date={1964},
  pages={vii+148},
}

\bib{dwyer-zabrodsky-maps-between-classifying}{article}{
  author={Dwyer, W.},
  author={Zabrodsky, A.},
  title={Maps between classifying spaces},
  conference={ title={Algebraic topology, Barcelona, 1986}, },
  book={ series={Lecture Notes in Math.}, volume={1298}, publisher={Springer, Berlin}, },
  date={1987},
  pages={106--119},
}

\bib{guillou-may-merling-cat-models-eq-class-space}{article}{
  author={Guillou, B. J.},
  author={May, J. P.},
  author={Merling, M.},
  title={Categorical models for equivariant classifying spaces},
  date={2012},
  eprint={arXiv:1201.5178},
}

\bib{lashof-equivariant-bundles}{article}{
  author={Lashof, R. K.},
  title={Equivariant bundles},
  journal={Illinois J. Math.},
  volume={26},
  date={1982},
  number={2},
  pages={257--271},
  issn={0019-2089},
}

\bib{lashof-may-gen-equivariant-bundles}{article}{
  author={Lashof, R. K.},
  author={May, J. P.},
  title={Generalized equivariant bundles},
  journal={Bull. Soc. Math. Belg. S\'er. A},
  volume={38},
  date={1986},
  pages={265--271 (1987)},
  issn={0037-9476},
}

\bib{lashof-may-segal-equivariant-abelian}{article}{
  author={Lashof, R. K.},
  author={May, J. P.},
  author={Segal, G. B.},
  title={Equivariant bundles with abelian structural group},
  conference={ title={Proceedings of the Northwestern Homotopy Theory Conference (Evanston, Ill., 1982)}, },
  book={ series={Contemp. Math.}, volume={19}, publisher={Amer. Math. Soc.}, place={Providence, RI}, },
  date={1983},
  pages={167--176},
}

\bib{luck-uribe-equivariant-principal-bundles}{article}{
  author={L\"uck, Wolfgang},
  author={Uribe, Bernardo},
  title={Equivariant principal bundles and their classifying spaces},
  journal={Algebr. Geom. Topol.},
  volume={14},
  date={2014},
  number={4},
  pages={1925--1995},
  issn={1472-2747},
}

\bib{may-classifying-fibrations}{article}{
  author={May, J. Peter},
  title={Classifying spaces and fibrations},
  journal={Mem. Amer. Math. Soc.},
  volume={1},
  date={1975},
  number={1, 155},
  pages={xiii+98},
  issn={0065-9266},
}

\bib{may-remarks-equivariant-bundles}{article}{
  author={May, J. P.},
  title={Some remarks on equivariant bundles and classifying spaces},
  note={International Conference on Homotopy Theory (Marseille-Luminy, 1988)},
  journal={Ast\'erisque},
  number={191},
  date={1990},
  pages={7, 239--253},
  issn={0303-1179},
}

\bib{montgomery-zippin-topological-transformation-groups}{book}{
  author={Montgomery, Deane},
  author={Zippin, Leo},
  title={Topological transformation groups},
  publisher={Interscience Publishers, New York-London},
  date={1955},
  pages={xi+282},
}

\bib{notbohm-maps-between-classifying}{article}{
  author={Notbohm, Dietrich},
  title={Maps between classifying spaces},
  journal={Math. Z.},
  volume={207},
  date={1991},
  number={1},
  pages={153--168},
  issn={0025-5874},
}


\end{biblist}
\end{bibdiv}
\end{document}